\documentclass[11pt, oneside]{amsart}
\usepackage{amstext}
\usepackage{amsmath, amssymb, amsthm}
\usepackage{amssymb}
\usepackage{mathrsfs}
\usepackage{enumerate}
\usepackage{xcolor}
\usepackage{comment}
\usepackage{hyperref}
\hypersetup{
    colorlinks=true,    
    linkcolor=blue,     
    citecolor=blue,    
    filecolor=blue,  
    urlcolor=blue       
}
\usepackage{graphicx}
\usepackage{mathtools}
\usepackage{accents}
\mathtoolsset{showonlyrefs}
\usepackage[a4paper, total={6in, 8.5in}]{geometry}
\usepackage{setspace}

\newtheorem{theorem}{Theorem}[section]
\newtheorem{lemma}[theorem]{Lemma}
\newtheorem{proposition}[theorem]{Proposition}
\newtheorem{corollary}[theorem]{Corollary}
\newtheorem{remark}[theorem]{Remark}

\DeclareMathOperator{\tr}{tr}

\newcommand{\circo}{\accentset{\circ}}
\newcommand\simtimes{\mathbin{%
    \stackrel{\sim}{\smash{\times}\rule{0pt}{0.8ex}}%
    }}

\title{The planarity estimate in dimensions five and six}
\author{Stephen Lynch}
\address{Department of Mathematics, King's College London, Strand, London, WC2R 2LS, UK}
\email{stephen.lynch@kcl.ac.uk}
\date{}

\begin{document}

\begin{abstract}
Naff's planarity estimate is a crucial tool for analysing singularities of quadratically pinched mean curvature flows in high codimension, showing that high-curvature regions become asymptotically codimension one. Naff established the estimate throughout the full Andrews--Baker pinching range in dimensions at least $7$, while in dimensions $5$ and $6$ a stronger pinching assumption was required. We close this remaining gap, proving the planarity estimate in dimensions $5$ and $6$ throughout the full Andrews--Baker range. The key new ingredient is a sharp Kato-type inequality relating the twisting of the principal normal to the gradient of the second fundamental form. 
\end{abstract}

\maketitle

\medskip

\noindent\textbf{AI usage statement.}
This project was developed with substantial assistance from Claude Fable 5.1. Claude found the sharp Kato-type inequality (see equation \eqref{sharp Bochner} below), which is the key new ingredient underlying the main results of the paper, and supplied an initial proof based on constrained optimisation. The author independently verified the inequality. The proof presented here was subsequently developed, in part through further discussion with Claude, with the aim of finding a more geometric argument. Claude also observed that the sharp Kato-type inequality is strong enough to remove the restriction in Naff's planarity estimate in dimensions $5$ and $6$. The author then verified this claim, carried out the remaining analysis, and developed the applications and interpretation of the result. Claude and ChatGPT were also used for mathematical discussion, checking, and editorial feedback. The author takes full responsibility for the final manuscript.

\section{Introduction}

A profound problem in geometric analysis is to understand the structure of singularities formed by the mean curvature flow. Many remarkable developments over the last decades have led to an extensive theory in the hypersurface setting. In higher codimension, however, the behaviour of the flow can be vastly more complex, and much less is understood. This comes down to the fact that, in codimension at least 2, the flow is described by a system of PDEs rather than a single scalar equation. Consequently, embeddedness is not preserved, so any general theory must account for self-intersections. In addition, the normal curvature tensor (which vanishes for a hypersurface) gives rise to nonlinear forcing terms, entering via Simons' identity, that influence the dynamics of the local geometry in a manner which is not easily understood. 

Developing a complete theory of singularities for the high-codimension mean curvature flow is, for now, probably not tractable. However, substantial progress has been made in a natural class introduced by Andrews--Baker \cite{Andrews--Baker}. They demonstrated that, in every dimension $n \geq 2$ and codimension $m \geq 2$, the quadratic pinching condition $|A|^2 < c|H|^2$ is preserved by the flow if $c \leq \frac{4}{3n}$. The threshold $\frac{4}{3n}$ is natural in the high codimension context; it arises from sharp pointwise inequalities for the normal curvature terms appearing in the evolution equation for the second fundamental form, and probably cannot be improved.

The first major progress concerning singularities of quadratically pinched flows was made by Andrews--Baker themselves \cite{Andrews--Baker}. They developed a complete picture in dimensions $n \in \{2,3,4\}$ for the full range of pinching constants $c \leq \frac{4}{3n}$, and in dimensions $n \geq 5$ assuming $c \leq \frac{1}{n-1}$: in these settings an initially pinched submanifold contracts to a point in finite time, becoming asymptotically spherical in the process. The next breakthrough came from Naff \cite{Naff}, who showed that at points of sufficiently large curvature the evolving submanifold becomes quantitatively close to a hypersurface in a suitable $(n+1)$-plane; this is the content of his planarity estimate. For $n \geq 7$, Naff proved the estimate throughout the full range $c \leq \frac{4}{3n}$, but in dimensions $n \in \{5,6\}$ his argument required the more restrictive hypothesis $c \leq \frac{3(n+1)}{2n(n+2)}$.

Naff's planarity estimate is extremely powerful, since it allows techniques from the hypersurface theory to be imported into the high codimension setting. Its only slightly unsatisfying feature is its restrictive pinching requirement in dimensions $n \in \{5,6\}$. This of course raises the question: Does a planarity estimate in fact hold in dimensions $n \in \{5,6\}$, for the full range of preserved pinching constants $c \leq \frac{4}{3n}$? Or, alternatively, could there exist pinched flows in these dimensions, with $c$ between $\frac{3(n+1)}{2n(n+2)}$ and $\frac{4}{3n}$, which form more complicated singularities that cannot be modelled on hypersurfaces? The latter scenario would seem quite strange, but is not ruled out by anything currently known.

The objective of the present article is to fill this gap, by proving a planarity estimate which holds for the full range of quadratic pinching conditions that are preserved by the flow, that is $|A|^2 < c|H|^2$ with $c \leq \frac{4}{3n}$, in the remaining dimensions $n \in \{5,6\}$. In order to state the estimate, we recall that whenever $|H| > 0$ the tensor $A$ splits into orthogonal components $A = h \nu + A^-$, where $\nu = H/|H|$.

\begin{theorem}[Planarity Estimate]\label{main}
Fix $n \geq 5$ and $m \geq 2$. Let $M_t$ be a family of closed $n$-dimensional immersions in $\mathbb{R}^{n+m}$ evolving by the mean curvature flow. If $Q = c|H|^2 - |A|^2$ is positive at the initial time for some $c \leq \frac{4}{3n}$, then we have 
\[
\max_{M_t} \frac{|A^-|^2}{Q^{1-\sigma}} \leq \max_{M_0} \frac{|A^-|^2}{Q^{1-\sigma}}
\]
for every $\sigma \in (0,\frac{1}{30}]$.
\end{theorem}

Since $Q \leq c|H|^2$, the planarity estimate implies that the scale-invariant quantity $|A^-|^2/|H|^2$ becomes small whenever $|H|^2$ is large enough relative to $M_0$.

\begin{corollary}
In the setting of Theorem~\ref{main}, the estimate 
\[
|A^-|^2 \leq C |H|^{2-2\sigma}
\]
holds at every point of $M_t$, where $C = C(n, \sigma, M_0)$.
\end{corollary}

It follows that all smooth blow-up limits at a singularity of the flow are hypersurfaces---see \cite[Proposition~2.5]{Naff}.

We have stated Theorem~\ref{main} in all dimensions $n \geq 5$, but let us repeat that it is only a genuine extension of Naff's work for $n \in \{5,6\}$. We omit dimensions $n \in \{2,3,4\}$ for two reasons: our proof does not work there, and, more importantly, in these dimensions the stronger umbilic estimate of Andrews--Baker \cite{Andrews--Baker} already forces singularity models to be spherical, rather than merely planar.

\subsection{The proof} As in \cite{Naff}, our proof of Theorem~\ref{main} proceeds by application of the parabolic maximum principle to the quantity $P := |A^-|^2/Q^{1-\sigma}$. Our key contribution is a better understanding of the gradient terms appearing in the evolution equation
\begin{align}\label{evol A-}
\frac{1}{2}(\partial_t - \Delta)|A^-|^2 &= |A^-\cdot A^-|^2 + \sum_{\alpha > 1}|R^\perp(\cdot,\cdot)\nu_\alpha|^2 - |\nabla A^-|^2 \notag \\
&- 2\langle \nabla h \otimes \nu, \nabla A^-\rangle + 2\frac{h_{ij}}{|H|}\langle \nabla |H| \otimes \nu, \nabla A^-_{ij}\rangle,
\end{align}
which was derived in \cite{Naff}.

The cross-terms on the second line need to be absorbed by the good Bochner term $|\nabla A^-|^2$ and further favourable terms that result from our dividing by $Q^{1-\sigma}$. Naive estimation of these cross-terms using Cauchy--Schwarz turns out to be insufficient. Instead, the first key observation is that, in both cross-terms, $\nabla A^-$ is paired with $\nu$, and hence can be replaced using the identity $\langle \nu, \nabla A^-\rangle = -\langle \nabla \nu, A^-\rangle$. The second key observation is that the tensors $\nabla \nu$ and 
\[
(\nabla A^-)^- = \nabla A^- - \langle \nu, \nabla A^-\rangle \nu
\]
are not independent; rather, they constrain each other through the Codazzi relations. This suggests that factors of $\nabla \nu$ can be absorbed using Cauchy--Schwarz and the second component of the Bochner term $|\nabla A^-|^2 = |\langle \nu, \nabla A^-\rangle|^2 + |(\nabla A^-)^-|^2$.

The new technical ingredient in our argument is a precise quantitative relationship between the tensors $\nabla \nu$ and $(\nabla A^-)^-$. This is the content of the following sharp Kato-type inequality, which holds for a general submanifold at any point where $|H| > 0$: after choosing an orthonormal frame where $h$ is diagonal with entries $\lambda_i$, we have
\begin{equation}\label{sharp Bochner}
|(\nabla A^-)^-|^2 \geq \frac{2n}{(n+2)(n-1)}\sum_i(|H| - \lambda_i)^2|\nabla_i \nu|^2.
\end{equation} 
Our proof of \eqref{sharp Bochner}, which can be found in Proposition~\ref{lb (nabla A-)-} below, relies on a judicious splitting of $(\nabla A^-)^-$ into orthogonal components, and sharp estimates for each of these components derived from the Codazzi relations. While estimates similar to \eqref{sharp Bochner} were employed in Naff's work, in hindsight one sees that these were not optimal: they mixed components of $(\nabla A^-)^-$ which need to be treated independently. This is precisely where the restriction $c \leq \frac{3(n+1)}{2n(n+2)}$ enters Naff's argument, while the new sharp version \eqref{sharp Bochner} allows us to carry through the proof of the planarity estimate for all $c \leq \frac{4}{3n}$.

While the inequality \eqref{sharp Bochner} appears to be new, it belongs to a long line of sharp Kato-type inequalities which have previously been exploited to prove beautiful theorems in geometric analysis. Huisken's inequality $|\nabla A|^2 \geq \frac{3}{n+2}|\nabla H|^2$ underlies both the convergence theorem for convex hypersurfaces under the mean curvature flow \cite{Huisken} and its generalisation to higher codimension \cite{Andrews--Baker}. Schoen--Simon--Yau employed the sharp inequality $|\nabla A|^2 \geq (1 + \tfrac{2}{n})\,|\nabla |A||^2$ (which requires $H=0$) to prove their curvature estimate for stable minimal hypersurfaces \cite{Schoen--Simon--Yau}. Further examples are plentiful; a systematic discussion of many of these can be found in \cite{CGH}.

\subsection{Structure} With the crucial new ingredient \eqref{sharp Bochner} at hand, we can largely proceed as in \cite{Naff}. However, rather than simply citing results from that paper, we have opted to give a complete proof of Theorem~\ref{main}. After fixing notation and collecting preliminary results in Section~\ref{sec prelim}, we analyse the reaction terms appearing in the evolution equation for $Q$ in Section~\ref{sec pinch}. In particular, we provide an alternative proof that the pinching condition $Q > 0$ is preserved for $c \leq \frac{4}{3n}$, tailored to our needs and in which the threshold $\frac{4}{3n}$ appears rather transparently. In Section~\ref{sec grad} we study the gradient terms in \eqref{evol A-}, and in particular establish \eqref{sharp Bochner}. Finally, in Section~\ref{sec plan} we carry out the proof of Theorem~\ref{main}.

\subsection{Applications} We now discuss some immediate applications of our extended planarity estimate. In recent work with Nguyen \cite{Lynch--Nguyen}, we developed a mean curvature flow with surgery for quadratically pinched submanifolds. In dimensions $n \geq 8$ the construction requires $c \leq \frac{1}{n-2}$, which is sharp for the kinds of surgeries performed, and in dimension $n = 7$ it applies to the full range of pinching constants $c \leq \frac{4}{21}$. However, in dimensions $n \in \{5,6\}$, although one expects to be able to perform surgery for the full range of pinching constants $c \leq \frac{4}{3n}$, we needed to impose $c \leq \frac{3(n+1)}{2n(n+2)}$ in order to invoke the planarity estimate from \cite{Naff}. Using Theorem~\ref{main} instead, with minor modifications to \cite{Lynch--Nguyen} (these are discussed in Section~\ref{sec surgery} below), we can now construct a flow with surgery for $n \in \{5,6\}$ and the full range of pinching conditions known to be preserved by the flow:

\begin{theorem}\label{surgery}
Fix $n \in \{5,6\}$ and $m \geq 2$. Let $M_0$ be a closed $n$-dimensional immersion in $\mathbb{R}^{n+m}$ satisfying $|A|^2 < c|H|^2$ with $c \leq \frac{4}{3n}$. Then there exists a mean curvature flow with surgery starting from $M_0$ and terminating after finitely many steps. 
\end{theorem}

As an immediate consequence of Theorem~\ref{surgery}, we can classify quadratically pinched immersions with $n \in \{5,6\}$ up to diffeomorphism. 

\begin{corollary}
Suppose $n \in \{5,6\}$ and $m \geq 2$. If $M_0$ is a closed $n$-dimensional immersion in $\mathbb{R}^{n+m}$ satisfying $|A|^2 < c|H|^2$ with $c \leq \frac{4}{3n}$, then it is diffeomorphic either to $\mathbb{S}^n$ or to a finite connected sum of the sphere bundles $\mathbb{S}^{n-1}\times\mathbb{S}^1$ and $\mathbb{S}^{n-1} \simtimes \mathbb{S}^1$.
\end{corollary}

Combined with the results of \cite{Andrews--Baker} and \cite{Lynch--Nguyen}, this completes the classification of quadratically pinched immersions with $c \leq \frac{4}{3n}$ for $2 \leq n \leq 8$.

\subsection{Previous work} To conclude this introduction, let us mention several other results concerning the mean curvature flow under quadratic pinching. Ancient solutions and solitons of higher codimension were studied in \cite{Risa--Sinestrari, Lynch--Nguyen_ancient, Naff_singularity, LNZ, IRR}, and singularity formation was studied in \cite{Lynch--Nguyen_convexity, Naff_canonical, BLL}. In the setting of surfaces in $\mathbb{R}^4$, more general quadratic pinching conditions were considered in \cite{Baker--Nguyen}. The articles \cite{Liu--Xu--Ye--Zhao, Liu--Xu--Zhao, Lei--Xu, Baker--Nguyen_sphere, Pipoli--Sinestrari, Nguyen--Vogiatzi, Vogiatzi} concern the flow of pinched submanifolds in non-Euclidean ambient spaces. An important precursor to \cite{Andrews--Baker} was \cite{Huisken_sphere}, which considered the mean curvature flow of quadratically pinched hypersurfaces in spheres. 

Outside the setting of quadratic pinching, Colding and Minicozzi have studied the interplay between entropy and codimension bounds for ancient solutions to the mean curvature flow \cite{Colding--Minicozzi}. They prove in particular that an ancient solution whose parabolic blow-down is a cylinder with multiplicity one must be planar. Altschuler showed in \cite{Altschuler} that singularity models of closed solutions to the high-codimension curve shortening flow are plane curves.

\section{Preliminaries}\label{sec prelim}

\subsection{Notation} Let $M$ be an $n$-dimensional immersed submanifold of $\mathbb{R}^{n+m}$. We write $A$ for its second fundamental form and $H$ for its mean curvature vector. When $|H| > 0$ we define $\nu = H/|H|$. We then set $h = \langle A, \nu\rangle$ and $A^- = A - h\nu$, and thus have the orthogonal decomposition $A = h\nu + A^-$. Note that $A^-$ is traceless. We will sometimes fix an orthonormal frame for the normal space, denoted $\nu_\alpha$, such that $\nu_1 = \nu$. In such a frame each of the components $A^-_\alpha = \langle A^-, \nu_\alpha\rangle$ is traceless. 

The induced connection on the tangent bundle of $M$ is given by $\nabla_X Y = (D_X Y)^\top$, where $D$ is the ambient Euclidean connection and $X$, $Y$ are tangent vector fields. The induced connection on the normal bundle is given by $\nabla_X V = (D_X V)^\perp$ where $X$ is tangent and $V$ is normal. The normal curvature tensor is 
\[
R^\perp(X,Y)V = \nabla_X\nabla_Y V - \nabla_Y\nabla_X V - \nabla_{[X,Y]}V,
\]
where $X,Y$ are tangent to $M$ and $V$ is normal. The Ricci equations assert that 
\[
\langle R^\perp(\cdot,\cdot)U, V\rangle 
=[\langle A,V\rangle, \langle A,U\rangle],
\]
where the right-hand side should be interpreted as the commutator of symmetric matrices after a choice of orthonormal basis for the tangent space. 

The notation $Q$ will always refer to $Q = c|H|^2 - |A|^2$ for some $c \leq \frac{4}{3n}$. When $Q$ is positive, $P$ refers to the ratio $P = |A^-|^2/Q^{1-\sigma}$, where $\sigma \in (0,1)$. 

\subsection{Evolution equations} In \cite{Andrews--Baker}, it was shown that along a solution of the mean curvature flow, we have
\[
\frac{1}{2}(\partial_t - \Delta)|A|^2 = |A\cdot A|^2 + |R^\perp|^2 - |\nabla A|^2.
\]
where $(A \cdot A)_{ijkl} = \langle A_{ij}, A_{kl}\rangle$, and 
\[
\frac{1}{2}(\partial_t-\Delta)|H|^2 = |\langle A, H\rangle|^2 -|\nabla H|^2. 
\]
Consequently,
\begin{align*}
\frac{1}{2}(\partial_t - \Delta)Q &= c|\langle A, H\rangle|^2 - |A\cdot A|^2 - |R^\perp|^2 - c|\nabla H|^2 + |\nabla A|^2.
\end{align*}
The evolution equation for $|A^-|^2$ derived in \cite{Naff} will also play an important role. It was stated in the introduction as equation \eqref{evol A-}. 

\subsection{Commutator estimate} The evolution equations for $Q$ and $|A^-|^2$ both contain terms which are quadratic in the normal curvature. These can be expressed in terms of commutators of normal components of $A$ via the Ricci equations. We will repeatedly estimate terms of this form using the following lemma.

\begin{lemma}\label{commutator}
Suppose $B$ and $C$ are two symmetric matrices. We then have
\[
2\langle B, C\rangle^2 + |[B,C]|^2 \leq 2|B|^2|C|^2.
\]
If in addition $C$ is traceless, then 
\[
2\langle B, C\rangle^2 + |[B,C]|^2 \leq 2|\circo{B}|^2|C|^2.
\]
\end{lemma}
\begin{proof} Diagonalising $B$ with eigenvalues $b_i$, we see that
\[
\langle B, C\rangle^2 = \bigg(\sum_{i,j} B_{ij}C_{ij}\bigg)^2 = \bigg(\sum_i b_i C_{ii}\bigg)^2 \leq |B|^2 \sum_i C_{ii}^2
\]
by Cauchy--Schwarz, and 
\[
|[B,C]|^2 = \sum_{i,j}(b_i - b_j)^2C_{ij}^2 = \sum_{i \not=j}(b_i - b_j)^2C_{ij}^2 \leq 2|B|^2 \sum_{i \not=j}(C_{ij})^2.
\]
Combining these inequalities,
\[
2\langle B, C\rangle^2 + |[B,C]|^2 \leq 2|B|^2 \bigg(\sum_i C_{ii}^2 + \sum_{i \not=j}(C_{ij})^2\bigg) = 2|B|^2|C|^2.
\]
If $C$ is traceless then
\[
\langle B, C\rangle = \langle\circo{B}, C\rangle \qquad \text{and} \qquad [B,C] = [\circo{B},C],
\]
so we can apply the previous inequality with $\circo{B}$ in place of $B$. 
\end{proof}

\subsection{Kato inequalities} The following inequality was proved by Huisken in \cite{Huisken}. Although it is now well known, we recall the argument, since related ideas play a role in the proof of Proposition~\ref{lb (nabla A-)-} below.

\begin{lemma}\label{Kato}
Let $E$ be a totally symmetric $(0,3)$ tensor on $TM$. We then have 
\[
|E|^2 \geq \frac{3}{n+2}|\tr(E)|^2.
\]
\end{lemma}
\begin{proof}
The tensor $E$ admits an orthogonal decomposition into a totally traceless component plus its trace part, which is
\[
\frac{1}{n+2}(\tr(E)_i g_{jk} + \tr(E)_k g_{ij} + \tr(E)_j g_{ki}).
\]
Expressing $|E|^2$ as the sum of the squared lengths of these two components, the claim follows, since
\[
\bigg|\frac{1}{n+2}(\tr(E)_i g_{jk} + \tr(E)_k g_{ij} + \tr(E)_j g_{ki})\bigg|^2 = \frac{3}{n+2}|\tr(E)|^2. \qedhere
\]
\end{proof}

Since $\nabla A$ is totally symmetric by the Codazzi equations, one can apply Lemma~\ref{Kato} to each normal component of $\nabla A$ to obtain the Kato inequality 
\[
|\nabla A|^2 \geq \frac{3}{n+2}|\nabla H|^2.
\]

As in \cite{Naff}, we will also find it useful to apply Lemma~\ref{Kato} to various totally symmetric projections of $\nabla A$. We define 
\begin{align*}
S &= \langle\nabla A, \nu\rangle\\
T &= (\nabla A)^- = \nabla A - S\nu\\
U &= \langle \nabla \circo{A}, \nu\rangle = \nabla \circo{h} + \langle\nabla A^-, \nu\rangle.
\end{align*}
The tensor $U$ controls $\nabla |H|$, and $T$ controls $|H|\nabla \nu$, through the following lemma (which recalls Equations 4.21 and 4.20 from \cite{Naff}).

\begin{lemma}\label{Kato projections}
At any point of $M$ where $|H| > 0$ we have
\begin{align*}
|\nabla |H||^2 &\leq \frac{n(n+2)}{2(n-1)}|U|^2,\\
|H|^2|\nabla\nu|^2 &\leq \frac{n+2}{3}|T|^2.
\end{align*}
\end{lemma}
\begin{proof}
Using $|S|^2 = |U|^2 + \frac{1}{n}|\langle \nabla H, \nu\rangle|^2$ and Lemma~\ref{Kato} we obtain
\[
|U|^2 = |S|^2 - \frac{1}{n}|\langle \nabla H, \nu\rangle|^2 \geq \bigg(\frac{3}{n+2}-\frac{1}{n}\bigg)|\langle \nabla H, \nu\rangle|^2 = \frac{2(n-1)}{n(n+2)}|\nabla |H||^2,
\]
which is the first claimed inequality. The second inequality is simply Lemma~\ref{Kato} applied to the tensor $T$. 
\end{proof}

The next statement is essentially the same as Lemma~4.7 in \cite{Naff}. 
\begin{lemma}\label{Bochner Q}
If $Q >0$ with $c \leq \frac{4}{3n}$ then we have 
\[
|\nabla A|^2 - c|\nabla H|^2 \geq \frac{5n-8}{6(n-1)}|U|^2 + \frac{5n-8}{n+2}(Q + |\circo{h}|^2 + |A^-|^2)|\nabla \nu|^2.
\]
\end{lemma}
\begin{proof}
Since $|S|^2 = |U|^2 + \frac{1}{n}|\nabla |H||^2$, using Lemma~\ref{Kato projections} and $c \leq \frac{4}{3n}$ we see that
\[
|S|^2 - c|\nabla |H||^2 \geq \bigg(1 - \bigg(c-\frac{1}{n}\bigg)\frac{n(n+2)}{2(n-1)}\bigg)|U|^2 \geq \frac{5n-8}{6(n-1)}|U|^2.
\]
Lemma~\ref{Kato projections} and $c \leq \frac{4}{3n}$ also imply that
\[
|T|^2 - c|H|^2|\nabla \nu|^2 \geq \frac{5n-8}{3n(n+2)}|H|^2 |\nabla \nu|^2.
\]
Since $|\nabla A|^2 = |S|^2 + |T|^2$ and $|\nabla H|^2 = |\nabla |H||^2 + |H|^2|\nabla \nu|^2$, the claim follows once we sum the previous two inequalities and insert
\[
|H|^2 = \bigg(c-\frac{1}{n}\bigg)^{-1}(Q + |\circo{h}|^2 + |A^-|^2) \geq 3n(Q + |\circo{h}|^2 + |A^-|^2). \qedhere
\]
\end{proof}

\subsection{Pinching and eigenvalues} To conclude this preliminary section we record an elementary lemma, which is useful for deriving bounds on the eigenvalues of $h$ from our quadratic pinching hypothesis. 

\begin{lemma}\label{eigenvalue est}
Let $h$ be a symmetric matrix. The eigenvalues of $h$ lie in the interval
\[
\left[\frac{\tr(h)}{n} - \sqrt{\frac{n-1}{n}}|\circo{h}|,\frac{\tr(h)}{n} + \sqrt{\frac{n-1}{n}}|\circo{h}|\right].
\]
\end{lemma}
\begin{proof}
Let us fix an arbitrary labelling $\lambda_i$ for the eigenvalues of $h$, where $i$ ranges from $1$ to $n$. First we observe that the Cauchy--Schwarz inequality implies
\[
\tr(h)^2 = \lambda_1^2 + 2\lambda_1\sum_{i=2}^{n}\lambda_i + \bigg(\sum_{i=2}^{n}\lambda_i\bigg)^2 \leq \lambda_1^2 + 2\lambda_1\sum_{i=2}^{n}\lambda_i + (n-1)\sum_{i=2}^n\lambda_i^2.
\]
Next we expand the right-hand side of the identity
\[
|\circo{h}|^2 = \frac{1}{2n}\sum_{i,j}(\lambda_i - \lambda_j)^2
\]
to obtain
\begin{align*}
|\circo{h}|^2 &= \frac{n-1}{n}\lambda_1^2 -\frac{2}{n}\lambda_1\sum_{i=2}^{n}\lambda_i + \frac{1}{n}\sum_{i=2}^n \lambda_i^2 + \frac{1}{2n}\sum_{i,j = 2}^n(\lambda_i-\lambda_j)^2.
\end{align*}
Combining these two facts, one arrives at the inequality
\[
\frac{n-1}{n}|\circo{h}|^2 - \frac{\tr(h)^2}{n^2} \geq \lambda_1^2 - \frac{2}{n}\lambda_1\tr(h).
\]
which asserts that $\lambda_1$ lies between the two roots of a quadratic. Solving for the roots yields
\[
\frac{\tr(h)}{n} - \sqrt{\frac{n-1}{n}}|\circo{h}| \leq \lambda_1 \leq \frac{\tr(h)}{n} + \sqrt{\frac{n-1}{n}}|\circo{h}|.
\]
Since our labelling of the eigenvalues was chosen arbitrarily, this chain of inequalities applies to all of them. 
\end{proof}

\section{Evolution of the pinching quantity}\label{sec pinch}

In this section we derive a lower bound for $(\partial_t - \Delta)Q$ along a solution to the mean curvature flow. In particular, the calculation will show that $Q > 0$ is preserved by the flow for $c \leq \frac{4}{3n}$. This is not a new result, and to derive the planarity estimate it would suffice to simply quote the calculations carried out in \cite{Andrews--Baker} and \cite{Naff}. However we have opted for an alternative derivation in which the threshold $\frac{4}{3n}$ appears rather naturally. 

We first estimate the reaction terms in the evolution equation for $|A|^2$. 
\begin{lemma}
At a point where $|H| > 0$ we have
\[
|A\cdot A|^2 + |R^\perp|^2 \leq |h|^4 + 4|\circo{h}|^2|A^-|^2 + 2|A^-|^4
\]
and hence
\[
\frac{1}{2}(\partial_t - \Delta)|A|^2 \leq |h|^4 + 4|\circo{h}|^2|A^-|^2 + 2|A^-|^4 - |\nabla A|^2.
\]
\end{lemma}
\begin{proof}
We first expand
\begin{align*}
|A\cdot A|^2 &=|h|^4 + 2\sum_{\alpha > 1} \langle h, A_\alpha\rangle^2 + \sum_{\alpha,\beta > 1} \langle A_{\alpha}, A_{\beta}\rangle^2.
\end{align*}
Next, using the Ricci equation in the form
\[
\langle R^\perp(\cdot, \cdot)\nu_\alpha,\nu_\beta\rangle = [A_\beta, A_\alpha],
\]
we obtain
\[
|R^\perp|^2 = \sum_{\alpha, \beta}|[A_\alpha, A_\beta]|^2 = 2\sum_{\alpha > 1} |[h, A_\alpha]|^2 + \sum_{\alpha, \beta > 1} |[A_\alpha,A_\beta]|^2.
\]
Putting these two identities together,
\begin{align*}
|A\cdot A|^2 + |R^\perp|^2 &= |h|^4 + 2\sum_{\alpha > 1} \langle h, A_\alpha\rangle^2 +2\sum_{\alpha > 1} |[h, A_\alpha]|^2\\
&\qquad + \sum_{\alpha,\beta > 1} \langle A_{\alpha}, A_{\beta}\rangle^2 + \sum_{\alpha, \beta > 1} |[A_\alpha,A_\beta]|^2.
\end{align*} 
We apply Lemma~\ref{commutator}, using the fact that $A_\alpha$ is traceless for each $\alpha > 1$, to bound
\begin{align*}
&2\sum_{\alpha > 1} \langle h, A_\alpha\rangle^2 + 2\sum_{\alpha > 1} |[h, A_\alpha]|^2 \leq 4|\circo{h}|^2|A^-|^2.
\end{align*}
We also apply Lemma~\ref{commutator} for each pair $\alpha, \beta > 1$ to bound
\begin{align*}
\sum_{\alpha,\beta > 1} \langle A_{\alpha}, A_{\beta}\rangle^2 + \sum_{\alpha, \beta > 1} |[A_\alpha,A_\beta]|^2 \leq 2|A^-|^4.
\end{align*}
Combining these inequalities, the claim follows.
\end{proof}

Next we estimate the reaction terms appearing in the evolution equation for $Q$.

\begin{lemma}\label{Reaction Q}
At a point where $|H| > 0$ we have
\begin{align*}
\frac{1}{2}(\partial_t - \Delta)Q &\geq |h|^2Q + 3|A^-|^2Q + \bigg(\frac{4}{n} - 3c\bigg)|A^-|^2 |H|^2 + |A^-|^4\\
&\qquad + |\nabla A|^2 - c|\nabla H|^2.
\end{align*}
\end{lemma}
\begin{proof}
Using the previous lemma we obtain the estimate
\begin{align*}
&c|h|^2|H|^2 - |A\cdot A|^2 - |R^\perp|^2\\
&\qquad \geq |h|^2(c|H|^2 - |h|^2) - 4|\circo{h}|^2|A^-|^2 - 2|A^-|^4.
\end{align*}
Inserting $|\circo{h}|^2 = |h|^2 - \frac{1}{n}|H|^2$, and collecting terms to find factors of $Q$, the right-hand side becomes
\begin{align*}
|h|^2Q + 3|A^-|^2Q + \bigg(\frac{4}{n} - 3c\bigg)|A^-|^2 |H|^2 + |A^-|^4.
\end{align*}
The claim follows. 
\end{proof} 

Combining Lemma~\ref{Reaction Q} with the Kato inequality $\frac{n+2}{3}|\nabla A|^2 \geq |\nabla H|^2$, we see that if $c\leq \frac{4}{3n}$ then at points where $|H| > 0$ we have
\[
\frac{1}{2}(\partial_t - \Delta)Q \geq |h|^2Q + 3|A^-|^2Q + \frac{5n-8}{9n}|\nabla A|^2.
\]
Since $Q > 0$ forces $|H| > 0$, the parabolic maximum principle implies that $Q >0$ is preserved if $c \leq \frac{4}{3n}$. Thus we recover the preservation of pinching theorem in \cite{Andrews--Baker}, as claimed at the beginning of this section.

\section{Estimates for gradient terms}\label{sec grad}
The good Bochner term $|\nabla A^-|^2$ which appears in \eqref{evol A-} can be split as
\[
|\nabla A^-|^2 = |\langle \nabla A^-, \nu\rangle|^2 + |(\nabla A^-)^-|^2.
\]
The main technical advance of the present article, which allows us to prove Theorem~\ref{main}, is a sharp lower bound for $|(\nabla A^-)^-|^2$
in terms of $h$ and $\nabla \nu$. We first state a general version of this new estimate in Proposition~\ref{lb (nabla A-)-}, before deriving the specific form we will need later in Corollary~\ref{cor1 lb (nabla A-)-}. 

To motivate the estimate, we recall that the tensor
\[
T = (\nabla A)^- = \nabla \nu \otimes h + (\nabla A^-)^-
\]
is totally symmetric because of the Codazzi equations. By comparing its different traces one observes (as in \cite[Proposition~2.5]{Naff}) that 
\[
(\nabla_j A^-_{ji})^- = |H|\nabla_i \nu - h_{ij} \nabla_j \nu.
\]
Passing to an orthonormal frame where $h$ is diagonal and applying Cauchy--Schwarz, we obtain the crude bound
\[
|(\nabla A^-)^-|^2 \geq \frac{1}{n} \sum_{i}(|H| - \lambda_i)^2|\nabla_i\nu|^2.
\]
This is the kind of inequality we need to exploit the Bochner term in \eqref{evol A-}, but since precise numerology will be important, it is desirable to have a sharp version in which the constant on the right-hand side is optimal. The following proposition provides this sharp inequality. 

\begin{proposition}\label{lb (nabla A-)-}
For an arbitrary $n$-dimensional submanifold of Euclidean space, at a point where $|H| > 0$, we have
\begin{align*}
|(\nabla A^-)^-|^2 \geq \frac{2n}{(n+2)(n-1)}\sum_i(|H| - \lambda_i)^2|\nabla_i \nu|^2
\end{align*} 
in any orthonormal frame such that $h$ is diagonal with entries $\lambda_i$.
\end{proposition}
\begin{proof}
To begin with, we split $(\nabla A^-)^-$ into its totally symmetric part and an orthogonal defect. Because of the Codazzi equations, we already know that 
\[
(\nabla A^-)^- =T - \nabla \nu \otimes h
\]
where $T$ is totally symmetric. Therefore, it suffices to split off the totally symmetric part of $\nabla \nu \otimes h$, which is given by
\[
B_{ijk} := \frac{1}{3}(h_{jk} \nabla_i \nu + h_{ij} \nabla_k \nu + h_{ki} \nabla_j \nu).
\]
The desired splitting of $(\nabla A^-)^-$ is then
\[
(\nabla A^-)^- =  (T - B) + (B - \nabla\nu\otimes h).
\]
Since the two components are orthogonal, we have
\[
|(\nabla A^-)^-|^2 = |T - B|^2 + |B - \nabla \nu \otimes h|^2.
\]

As in Lemma~\ref{Kato} we can perform a further orthogonal decomposition of $T-B$ into its trace and trace-free parts. Denoting the latter by $Z$, this gives
\[
|(\nabla A^-)^-|^2 = \frac{3}{n+2}|\tr(T) - \tr(B)|^2 + |B - \nabla \nu \otimes h|^2 + |Z|^2.
\]

Next we compute the two terms $|\tr(T) - \tr(B)|^2$ and $|B - \nabla \nu \otimes h|^2$ using an orthonormal frame in which $h$ is diagonal with entries $\lambda_i$. Since
\[
\tr(T)_i = |H|\nabla_i \nu, \qquad \tr(B)_i = \frac{1}{3}(|H|\nabla_i \nu + 2h_{ij}\nabla_j\nu),
\]
we find that 
\[
\tr(T)_i - \tr(B)_i = \frac{2}{3} |H|\nabla_i \nu - \frac{2}{3}h_{ij}\nabla_j \nu,
\]
and hence
\[
|\tr(T) - \tr(B)|^2 = \frac{4}{9}\sum_i(|H| - \lambda_i)^2|\nabla_i \nu|^2.
\]
In addition, we have
\[
|B - \nabla \nu \otimes h|^2 = \frac{2}{3}|h|^2 |\nabla \nu|^2 - \frac{2}{3}\sum_i \lambda_i^2 |\nabla_i \nu|^2 = \frac{2}{3}\sum_j\bigg(|\nabla_j \nu|^2\sum_{i\not= j} \lambda_i^2\bigg).
\]
It follows that
\[
|(\nabla A^-)^-|^2 = \frac{12}{9(n+2)}\sum_i(|H| - \lambda_i)^2|\nabla_i \nu|^2 + \frac{2}{3}\sum_j\bigg(|\nabla_j \nu|^2\sum_{i\not= j} \lambda_i^2\bigg) + |Z|^2.
\]

To conclude we use the Cauchy--Schwarz inequality to bound
\[
\sum_j\bigg(|\nabla_j \nu|^2\sum_{i\not= j} \lambda_i^2\bigg) \geq \frac{1}{(n-1)}\sum_j (|H|-\lambda_j)^2|\nabla_j \nu|^2,
\]
and thus obtain
\[
|(\nabla A^-)^-|^2 \geq \frac{2n}{(n+2)(n-1)}\sum_i(|H| - \lambda_i)^2|\nabla_i \nu|^2 + |Z|^2.
\] 
The claim follows after we discard the nonnegative term $|Z|^2$. 
\end{proof}
\begin{remark}
To see that the constant in Proposition~\ref{lb (nabla A-)-} is sharp, we describe a configuration in which all of the inequalities in the proof are saturated. This occurs when $h = \operatorname{diag}(0,1,\dots, 1)$; the one-form $\tr(T)$, and hence $\nabla \nu = \tr(T)/|H|$, is concentrated along $e_1$; and the trace-free part of $T$ coincides with that of $B$. To see that such a configuration really occurs, note that we can prescribe $h$ and $T$ at a point, and the two requirements on $T$ concern its trace- and trace-free parts, which are independent of each other and so can be prescribed individually. 
\end{remark}

Next we derive a straightforward consequence of Proposition~\ref{lb (nabla A-)-}. The key point is to estimate the terms $|H| - \lambda_i$ appearing there so that they can easily be controlled using our quadratic pinching hypothesis. 

\begin{corollary}\label{cor1 lb (nabla A-)-}
For an arbitrary $n$-dimensional submanifold of Euclidean space, at a point where $|h|^2 < |H|^2$, we have
\begin{align*}
|(\nabla A^-)^-|^2 &\geq \frac{2}{n+2}\bigg(\sqrt{\frac{n-1}{n}}|H| - |\circo{h}|\bigg)^2|\nabla \nu|^2.
\end{align*} 
\end{corollary}
\begin{proof}
We consider the sum 
\[
\sum_i(|H| - \lambda_i)^2|\nabla_i \nu|^2
\]
which appears in Proposition~\ref{lb (nabla A-)-}. Using Lemma~\ref{eigenvalue est}, we see that $|H| - \lambda_i$ can be no smaller than
\[
\frac{n-1}{n}|H| - \sqrt{\frac{n-1}{n}}|\circo{h}|,
\]
which is a positive number since $|h|^2 < |H|^2$ by assumption. It follows that
\[
\sum_i(|H| - \lambda_i)^2|\nabla_i \nu|^2 \geq \frac{n-1}{n}\bigg(\sqrt{\frac{n-1}{n}}|H| - |\circo{h}|\bigg)^2|\nabla \nu|^2,
\]
and this yields the claim when combined with Proposition~\ref{lb (nabla A-)-}. 
\end{proof}

To conclude this section we record a technical lemma, which is a consequence of Corollary~\ref{cor1 lb (nabla A-)-}, and will be used in the sequel. 

\begin{corollary}\label{cor2 lb (nabla A-)-}
Assuming $n \geq 5$ and $Q = c|H|^2 - |A|^2 > 0$ with $c \leq \frac{4}{3n}$, we have
\[\frac{1}{5n-8}(\sqrt{6(n-1)}\sqrt{Q} + \sqrt{n+2}|\circo{h}|)^2|\nabla \nu|^2-\theta^2|(\nabla A^-)^-|^2 \leq 0\]
whenever $1 \geq \theta \geq \sqrt{14/15}$.
\end{corollary}
\begin{proof}
Let us define 
\[
a = \frac{1}{\sqrt{5n-8}}(\sqrt{6(n-1)}\sqrt{Q} + \sqrt{n+2}|\circo{h}|)\]
and 
\[
b = \sqrt{\frac{2}{n+2}}\bigg(\sqrt\frac{n-1}{n}|H| - |\circo{h}|\bigg).
\]
Since $|(\nabla A^-)^-|^2 \geq b^2|\nabla \nu|^2$ by Corollary~\ref{cor1 lb (nabla A-)-}, the claim follows if we can establish that
\[
a^2 - \theta^2 b^2 \leq 0.
\]
Clearly $a \geq 0$, and also $b \geq 0$ since our pinching hypothesis implies
\[|\circo{h}|^2 = |h|^2 - \frac{1}{n}|H|^2 < \frac{1}{3n}|H|^2 < \frac{n-1}{n}|H|^2,\]
so it will suffice to show that
\[
a - \theta b \leq 0. 
\]

We have 
\[
\frac{(a - \theta b)}{|H|} = p\frac{\sqrt{Q}}{|H|} + q\frac{|\circo{h}|}{|H|} - \theta \sqrt{\frac{2(n-1)}{n(n+2)}},
\]
where 
\[
p = \sqrt{\frac{6(n-1)}{5n-8}}, \qquad q = \sqrt{\frac{n+2}{5n-8}} + \theta \sqrt{\frac{2}{n+2}}.
\]
Using the Cauchy--Schwarz inequality and $Q + |\circo{h}|^2 \leq \frac{1}{3n}|H|^2$, we find that
\[
p\frac{\sqrt{Q}}{|H|} + q\frac{|\circo{h}|}{|H|} \leq \sqrt{p^2 + q^2}\sqrt{\frac{Q + |\circo{h}|^2}{|H|^2}} \leq \frac{1}{\sqrt{3n}}\sqrt{p^2 + q^2},
\]
so the inequality $a - \theta b \leq 0$ will follow if we can establish
\[
\frac{1}{\sqrt{3n}}\sqrt{p^2 + q^2} \leq \theta \sqrt{\frac{2(n-1)}{n(n+2)}},
\]
or equivalently
\[
p^2 + q^2 \leq \frac{6(n-1)}{n+2}\theta^2.
\]

It is not difficult to see that the left-hand side of the previous inequality is decreasing in $n$, whereas the right-hand side is increasing. Therefore, it suffices to verify the inequality for $n = 5$. This amounts to checking that
\[
\frac{31}{17} + 2\sqrt{\frac{2}{17}}\theta  + \frac{2}{7}\theta^2 \leq \bigg(3 + \frac{3}{7}\bigg)\theta^2,
\]
which is easily verified if $1 \geq \theta^2 \geq 14/15$, using for example $\frac{31}{17} \leq 2$ and $2\sqrt{\frac{2}{17}} \leq \frac{4}{5}$.
\end{proof}

\begin{remark}
Corollary~\ref{cor2 lb (nabla A-)-} fails if $n \in \{2,3,4\}$. This is the reason why our proof of Theorem~\ref{main} requires $n \geq 5$. 
\end{remark}

\section{Proof of the planarity estimate}\label{sec plan}

Consider a family of closed $n$-dimensional immersions in Euclidean space which are evolving under the mean curvature flow. We assume that $Q = c|H|^2 - |A|^2$ is positive with $c \leq \frac{4}{3n}$. The ratio $P = |A^-|^2/Q^{1-\sigma}$ satisfies
\begin{align*}
(\partial_t - \Delta) P &= \frac{1}{Q^{1-\sigma}}\bigg((\partial_t - \Delta)|A^-|^2 - (1-\sigma)\frac{|A^-|^2}{Q}(\partial_t - \Delta)Q\bigg)\\
&\qquad -\sigma(1-\sigma)P\frac{|\nabla Q|^2}{Q^{2}} + \frac{2}{Q^{1-\sigma}}\langle\nabla Q^{1-\sigma}, \nabla P\rangle.
\end{align*}

Theorem~\ref{main} will follow from the parabolic maximum principle once we establish that, for $\sigma \in (0,1)$ suitably small,
\[
(\partial_t - \Delta)|A^-|^2 - (1-\sigma)\frac{|A^-|^2}{Q}(\partial_t - \Delta)Q \leq 0.
\]
Using our estimate from Lemma~\ref{Reaction Q}, that is
\[
\frac{1}{2}(\partial_t - \Delta) Q \geq |h|^2 Q + 3|A^-|^2Q + |\nabla A|^2 - c|\nabla H|^2, 
\]
and the evolution equation for $|A^-|^2$ recorded in \eqref{evol A-}, we obtain
\[
(\partial_t - \Delta)|A^-|^2 - (1-\sigma)\frac{|A^-|^2}{Q}(\partial_t - \Delta)Q \leq 2 R + 2G,
\]
where 
\[
R := |A^-\cdot A^-|^2 + \sum_{\alpha > 1}|R^\perp(\cdot,\cdot)\nu_\alpha|^2 - (1-\sigma)|A^-|^2(|h|^2 + 3|A^-|^2)
\]
and 
\begin{align*}
G &:= - |\nabla A^-|^2 - 2\langle \nabla h \otimes \nu, \nabla A^-\rangle + 2\frac{h_{ij}}{|H|}\langle \nabla |H| \otimes \nu, \nabla A^-_{ij}\rangle\\
&\qquad - (1-\sigma)\frac{|A^-|^2}{Q}(|\nabla A|^2 - c|\nabla H|^2).
\end{align*}
Our task is therefore reduced to establishing $R \leq 0$ and $G \leq 0$. 

We first estimate the reaction terms $R$. 
\begin{proposition}\label{planarity reaction}
Assuming $Q > 0$ with $c \leq \frac{4}{3n}$ we have
\[
R \leq (\sigma - 1/2)|h|^2|A^-|^2 + (3\sigma - 5/2)|A^-|^4 - \frac{3}{2}|A^-|^2 Q.
\]
\end{proposition}
\begin{proof}
Let us fix an orthonormal frame for the normal space, denoted $\nu_\alpha$, such that $\nu_1 = \nu$. Using the Ricci equations and $[h, A_\alpha^-] = [\circo{h},A_\alpha^-]$ we expand
\begin{align*}
&|A^-\cdot A^-|^2 + \sum_{\alpha > 1}|R^\perp(\cdot,\cdot)\nu_\alpha|^2\\
&\qquad = \sum_{\alpha,\beta > 1}\langle A^-_\alpha, A^-_\beta\rangle^2 + \sum_{\alpha,\beta > 1}|[A_\alpha^-, A_\beta^-]|^2 + \sum_{\alpha > 1}|[\circo{h},A_\alpha^-]|^2.
\end{align*}
We estimate the first two terms using Lemma~\ref{commutator}, and the final term using Cauchy--Schwarz, in order to obtain
\[
|A^-\cdot A^-|^2 + \sum_{\alpha > 1}|R^\perp(\cdot,\cdot)\nu_\alpha|^2 \leq 2|A^-|^4 + 2|\circo{h}|^2|A^-|^2. 
\]
Inserting
\begin{align*}
|\circo{h}|^2 &= \bigg(1-\frac{1}{cn}\bigg)|h|^2 - \frac{1}{cn} |A^-|^2 - \frac{1}{cn}Q \leq \frac{1}{4}|h|^2 - \frac{3}{4}|A^-|^2 - \frac{3}{4}Q,
\end{align*}
this becomes
\[
|A^-\cdot A^-|^2 + \sum_{\alpha > 1}|R^\perp(\cdot,\cdot)\nu_\alpha|^2 \leq \frac{1}{2}|A^-|^4 + \frac{1}{2}|h|^2|A^-|^2 -\frac{3}{2}|A^-|^2 Q.
\]
The claim follows by inserting this inequality into the definition of $R$. 
\end{proof}

Next we bound the gradient terms $G$.

\begin{proposition}\label{planarity gradient}
Assuming $n \geq 5$, and that $Q > 0$ with $c \leq \frac{4}{3n}$, we have
\[
G \leq -\frac{1}{29}|\nabla A^-|^2
\]
provided $\sigma \leq 1/30$.
\end{proposition}
\begin{proof}
We begin by estimating the two cross terms appearing in $G$. These are 
\[
- 2\langle \nabla h \otimes \nu, \nabla A^-\rangle \qquad \text{and} \qquad 2\frac{h_{ij}}{|H|}\langle \nabla |H| \otimes \nu, \nabla A^-_{ij}\rangle.
\]
Since $\nabla_i A_{jk}^-$ is traceless in $(j,k)$ we have
\[
-2\langle \nabla h \otimes \nu, \nabla A^-\rangle = -2\nabla_i\circo{h}_{jk}\langle \nu, \nabla_i A^-_{jk}\rangle.
\]
Inserting $U = \nabla \circo{h} + \langle \nabla A^-, \nu\rangle$, and using $\langle \nu, A^-\rangle = 0$ to express
\[
\langle \nabla A^-, \nu \rangle = -\langle A^-, \nabla \nu\rangle,
\]
this becomes
\begin{align*}
-2\langle \nabla h \otimes \nu, \nabla A^-\rangle =2U_{ijk} \langle \nabla_i \nu, A^-_{jk}\rangle + 2|\langle \nabla A^-, \nu\rangle|^2.
\end{align*}
Using the Cauchy--Schwarz inequality, we obtain
\[
-2\langle \nabla h \otimes \nu, \nabla A^-\rangle \leq 2|A^-||U| |\nabla \nu|  + 2|\langle \nabla A^-, \nu\rangle|^2.
\]
Again by Cauchy--Schwarz, the other cross term satisfies
\begin{align*}
2\frac{h_{ij}}{|H|}\langle \nabla |H| \otimes \nu, \nabla A^-_{ij}\rangle &= -2\frac{\circo{h}_{ij}}{|H|}\nabla_k|H| \langle \nabla_k \nu, A^-_{ij}\rangle\\
&\leq 2\frac{|A^-|}{|H|}|\nabla |H|||\circo{h}||\nabla \nu|.
\end{align*}
After inserting these inequalities into the definition of $G$, and decomposing
\[
|\nabla A^-|^2 = |\langle \nabla A^-, \nu\rangle |^2 + |(\nabla A^-)^-|^2,
\]
we arrive at 
\begin{align*}
G &\leq |\langle \nabla A^-, \nu\rangle |^2 - |(\nabla A^-)^-|^2 + 2|U| |A^-||\nabla \nu|  + 2\frac{|A^-|}{|H|}|\nabla |H|||\circo{h}||\nabla \nu|\\
&\qquad - (1-\sigma)\frac{|A^-|^2}{Q}(|\nabla A|^2 - c|\nabla H|^2).
\end{align*}

In the next step we use the Bochner term coming from $Q$, that is $|\nabla A|^2 - c|\nabla H|^2$, to absorb the component $|\langle \nabla A^-, \nu\rangle |^2$ and as much of the two cross-terms as possible. First note that because of Lemma~\ref{Kato projections} and $\sqrt{3n}\sqrt{Q} \leq |H|$ (which requires $c \leq \frac{4}{3n}$) we have
\[
2\frac{|A^-|}{|H|}|\nabla |H|||\circo{h}||\nabla \nu| \leq 2\sqrt{\frac{n+2}{6(n-1)}} \frac{|A^-|}{\sqrt{Q}}|U||\circo{h}||\nabla \nu|.
\]
Combining this with $2xy \leq \gamma x^2 + \gamma^{-1} y^2$, we find that the two cross-terms satisfy the bound
\begin{align*}
&2|U| |A^-||\nabla \nu| + 2\frac{|A^-|}{|H|}|\nabla |H|||\circo{h}||\nabla \nu|\\
&\qquad \leq 2\frac{|A^-|}{\sqrt{Q}}|U|\bigg(\sqrt{Q} + \sqrt{\frac{n+2}{6(n-1)}}|\circo{h}|\bigg)|\nabla \nu|\\
&\qquad \leq \gamma \frac{|A^-|^2}{Q}|U|^2 + \gamma^{-1}\bigg(\sqrt{Q} + \sqrt{\frac{n+2}{6(n-1)}}|\circo{h}|\bigg)^2|\nabla \nu|^2,
\end{align*}
where we are free to choose the parameter $\gamma$. In order to maximally exploit the lower bound
\begin{align*}
\frac{|A^-|^2}{Q}(|\nabla A|^2 - c|\nabla H|^2) &\geq \frac{5n-8}{6(n-1)}\frac{|A^-|^2}{Q}|U|^2 + \frac{5n-8}{n+2}|A^-|^2|\nabla \nu|^2\\
&\geq \frac{5n-8}{6(n-1)}\frac{|A^-|^2}{Q}|U|^2 + \frac{5n-8}{n+2}|\langle \nabla A^-, \nu\rangle|^2
\end{align*}
which is an immediate consequence of Lemma~\ref{Bochner Q}, we choose $\gamma = (1-\sigma) \frac{5n-8}{6(n-1)}$. Doing so, we arrive at the estimate
\begin{align*}
G &\leq \bigg(1 - (1-\sigma)\frac{5n-8}{n+2}\bigg)|\langle \nabla A^-, \nu\rangle |^2 - |(\nabla A^-)^-|^2\\
&\qquad + \frac{1}{(1-\sigma)(5n-8)}\bigg(\sqrt{6(n-1)}\sqrt{Q} + \sqrt{n+2}|\circo{h}|\bigg)^2|\nabla \nu|^2.   
\end{align*}

Finally, we use the term $|(\nabla A^-)^-|^2$ to absorb the nonnegative term on the last line of the previous inequality. This is the step which uses the sharp Kato-type estimate from Proposition~\ref{lb (nabla A-)-}, via Corollary~\ref{cor2 lb (nabla A-)-}. Indeed, Corollary~\ref{cor2 lb (nabla A-)-} implies
\[
\frac{1}{(5n-8)}\bigg(\sqrt{6(n-1)}\sqrt{Q} + \sqrt{n+2}|\circo{h}|\bigg)^2|\nabla \nu|^2 - \frac{14}{15}|(\nabla A^-)^-|^2 \leq 0.
\]
Inserting this above, we see that 
\begin{align*}
G &\leq \bigg(1 - (1-\sigma)\frac{5n-8}{n+2}\bigg)|\langle \nabla A^-, \nu\rangle |^2 - \bigg(1-\frac{14}{15}(1-\sigma)^{-1}\bigg)|(\nabla A^-)^-|^2.  
\end{align*}
Both coefficients on the right are at most $-1/29$ if $n \geq 5$ and $\sigma \leq 1/30$. This completes the proof. 
\end{proof}

We now combine these pieces to establish Theorem~\ref{main}. 

\begin{proof}[Proof of Theorem~\ref{main}]
Suppose $\sigma$ has been fixed in $(0, 1/30]$. We recall, from the discussion at the beginning of this section, that
\begin{align*}
(\partial_t - \Delta) P &\leq \frac{1}{Q^{1-\sigma}}(2R + 2G)-\sigma(1-\sigma)P\frac{|\nabla Q|^2}{Q^{2}} + \frac{2}{Q^{1-\sigma}}\langle\nabla Q^{1-\sigma}, \nabla P\rangle.
\end{align*}
Since $R \leq 0$ by Proposition~\ref{planarity reaction}, and $G \leq 0$ by Proposition~\ref{planarity gradient}, the claim follows as a consequence of the parabolic maximum principle. 
\end{proof}

\section{Surgery}\label{sec surgery}

In this final section, we describe the modifications that need to be made in \cite{Lynch--Nguyen} in order to establish Theorem~\ref{surgery}. First, in the definition of the surgery class in \cite[Subsection~2.2]{Lynch--Nguyen}, the pinching constant $c_n$ should be set to $\frac{4}{3n}$ in dimensions $n \in \{5,6\}$. To carry out the proof of \cite[Theorem~5.1]{Lynch--Nguyen} under this more general pinching hypothesis, one must simply substitute Theorem~\ref{main} of the present article for Naff's version of the planarity estimate. Let us also note that the requirement on $\sigma$ in \cite[Theorem~5.1]{Lynch--Nguyen} can now be replaced by $\sigma \in (0,1/30]$ for all $n \geq 5$. The setup of \cite[Theorem~5.3]{Lynch--Nguyen} also needs to be modified---the proof there still works provided $c$ lies in the range $[\frac{1}{2}(\frac{4}{3n} + \frac{1}{n-1}), \frac{4}{3n}]$ in dimensions $n \in \{5,6\}$. The rest of the analysis in \cite{Lynch--Nguyen} goes through without change.

\end{document}